\documentclass[12pt,a4paper,reqno]{amsart}

\usepackage{anysize}
\usepackage[utf8]{inputenc}
\usepackage{amsmath}
\usepackage{amssymb}
\usepackage{array}

\newtheorem{theorem}{Theorem}[section]

\newtheorem{conjecture}[theorem]{Conjecture}

\begin{document}

\title{Ramanujan series with a shift}
\author{Jesús Guillera}
\address{Department of Mathematics, University of Zaragoza, 50009 Zaragoza, SPAIN}
\email{jguillera@gmail.com}

\dedicatory{In memory of Jon Borwein, \\ the main driver and master of the Experimental Mathematics.}



\date{}

\maketitle

\begin{abstract}
We consider an extension of the Ramanujan series with a variable $x$. If we let $x=x_0$, we call the resulting series: ``Ramanujan series with the shift $x_0$". Then, we relate these shifted series to some $q$-series and solve the case of level $4$ with the shift $x_0=1/2$. Finally, we indicate a possible way towards proving some patterns observed by the author corresponding to the levels $\ell=1, 2, 3$ and the shift $x_0=1/2$. 
\end{abstract}

\section{shift and upside-down transformations}

We call a shift to a transformation which consist of applying the substitution $n \to n + x_0$ inside a series, and we say that $x_0$ is the shift. For example, the series 
\begin{equation}\label{rama}
\sum_{n=0}^{\infty} z^n  \frac{ \left( \frac{1}{2} \right)_{n}
\left( \frac{1}{s}  \right)_{n}\left( \frac{s-1}{s} \right)_{n}}{ (1)_{n}^3}(a+bn),
\end{equation}
shifted $x_0$ becomes
\begin{equation}\label{rama-G-x0}
\sum_{n=0}^{\infty} z^{n+x_0}  \frac{ \left( \frac12+x_0 \right)_n\left(  \frac{1}{s}+x_0 \right)_n \left( \frac{s-1}{s}+x_0 \right)_n}{ (1+x_0)_n^3} (a+b(n+x_0)).
\end{equation}
multiplied by a factor which does not depend on $n$ (we will ignore that factor). An upside down transformation consist of the substitution $n \to -n$. That is
\begin{equation}\label{rama-ud}
\sum_{n=1}^{\infty} z^{-n}  \frac{ \left( \frac{1}{2} \right)_{-n}
\left( \frac{1}{s}  \right)_{-n}\left( \frac{s-1}{s} \right)_{-n}}{ (1)_{-n}^3}(a-bn),
\end{equation}
understanding the rising factorials in the way indicated below 
\[
(a)_{-n} \rightarrow \frac{(-1)^n}{(1-a)_n} \, \, \text{if $a\neq 1$ \, and \,}  \quad (1)_{-n} \rightarrow \frac{n(-1)^n}{(1)_n}.
\]
These substitutions are justified because they preserve formally the recurrence equation $\Gamma(x+1)=x \, \Gamma(x)$; see the duality property \cite[Chapter 7]{PWZ} and the application shown in \cite[Section 4]{Gui-WZ-diver}, and see \cite{Gui-Rog} for the analytic interpretation. If $|z|>0$ we understand the ``divergent" series (\ref{rama}) as its analytic continuation, and if $|z|<0$ we interpret the ``divergent" series (\ref{rama-ud}) in the same way. While in \cite{Gui-Rog} we have studied the ``upside-down" transformation, in this paper we consider the  transformations with a shift. In \cite{Gui-Rog} we prove that the upside-down transformation modify the value of the modular variable $q$. Here we will see that a shift do not modify it.

The following kind of series for $1/\pi$:
\begin{equation}\label{ramanujan}
\sum_{n=0}^{\infty} z^n  \frac{ \left( \frac{1}{2} \right)_{n}
\left( \frac{1}{s}  \right)_{n}\left( \frac{s-1}{s} \right)_{n}}{ (1)_{n}^3}(a+bn)=\frac{1}{\pi},
\end{equation}
where $s \in \{2,3,4,6 \}$ can be parametrized with a modular function $z=z_{\ell}(q)$, and two modular forms  $b=b_{\ell}(q)$ and $a=a_{\ell}(q)$ of weight $2$. It is known that the level of these functions is $\ell= 1, 2, 3, 4$ for $s \!=\! 6, 4, 3, 2$ respectively, and that for $q=\pm e^{-\pi \sqrt{r}}$  with $r \in \mathbb{Q^{+}}$ the values of $z$, $b$, $a$ are algebraic reals (the sign $+$ corresponds to series of positive terms and the sign $-$ to alternating series). In these cases the series (\ref{ramanujan}) are named as Ramanujan-type series, in honour to the Indian genius Srinivasa Ramanujan who gave $17$ examples of them. If we want to consider algebraic complex solutions, then we let $q=e^{2\pi i \tau}$, where $\tau$ is a quadratic irrational with ${\rm Im}(\tau)>0$.
In this paper we are mainly interested in the evaluations of (\ref{completa-zq}) for those special values of $q$ and $x=1/2$. We will use the following theorems:

\begin{theorem}\label{teor-y-rama} 
Let 
\begin{equation}\label{y-x}
F_{\ell}(x,z)=\sum_{n=0}^{\infty} z^{n+x}  \frac{ \left( \frac12 + x  \right)_n
\left(  \frac{1}{s} + x \right)_n \left( \frac{s-1}{s} + x \right)_n}{ (1+x)_n^3 }, \qquad \frac{F_{\ell}(x,z)}{F_{\ell}(0,z)}=\phi(q),
\end{equation}
and
\begin{equation}\label{rama-G-x}
G_{\ell}(x,z)=\sum_{n=0}^{\infty} z^{n+x}  \frac{ \left( \frac12+x \right)_n\left(  \frac{1}{s}+x \right)_n \left( \frac{s-1}{s}+x \right)_n}{ (1+x)_n^3} (a+b(n+x)),
\end{equation}
where $z=z(q)$, $b=b(q)$ and $a=a(q)$ are the functions mentioned before. Then we have
\begin{equation}\label{completa-zq}
G_{\ell}(x,q)=\frac{1}{\pi} \left( \phi(q) - \ln|q| \, q \frac{d \phi(q)}{dq} \right).
\end{equation}
\end{theorem}

\begin{proof}
It is a particular case of \cite[Proposition 2]{Gui-Rog}.
\end{proof}

\begin{theorem}\label{teor-yang}
The following identity holds:
\begin{equation}\label{for-yang}
\left( q \frac{d}{dq} \right)^3 \phi(q)=\frac{x^3 z^x}{\sqrt{1-z}} \left( \frac{q}{z} \frac{dz}{dq} \right)^2 = x^3 F_{\ell}^2(0,q) \; \sqrt{1-z} \; \sqrt{z}.
\end{equation}
\end{theorem}
\begin{proof}
The differential operator:
\[
\mathcal{D}=\left( z\frac{d}{dz} \right)^3 - z \left( z\frac{d}{dz}+\frac{1}{2} \right)\left( z\frac{d}{dz}+ \frac{1}{s} \right)\left( z\frac{d}{dz}+\frac{s-1}{s} \right),
\]
annihilates $F_{\ell}(0,z)$, and in  \cite{Gui-matrix} we proved that $\mathcal{D} F_{\ell}(x,z)=x^3 z^x$. As $F(0,q)$ is a modular form such that $\mathcal{D} F(0,z)=0$ we can apply \cite[Lemma 1]{yang}, and as 
\[
\mathcal{D}=(1-z)\left( z\frac{d}{dz} \right)^3+\cdots,
\]
we have
\[ 
\left( q \frac{d}{dq} \right)^3 \frac{F_{\ell}(x,q)}{F_{\ell}(0,q)}=\frac{\mathcal{D} F_{\ell}(x,z)}{F_{\ell}(0,z) (1-z)} \left( \frac{q}{z} \frac{dz}{dq} \right)^3,
\]
Finally, using \cite[eq. 2.34]{Gui-matrix} we complete the proof.
\end{proof}

\section{Ramanujan series of level $4$ with a shift}

This is motivated by the evaluations found in \cite{Gui-Rog} by observing that when $s=2$, a shift of $x=1/2$ of a convergent Ramanujan-type series is equivalent to the upside-down of a related ``divergent" Ramanujan-type series.

\subsection{Ramanujan series of level $4$ with the shift $1/2$}

\begin{theorem}
Case $s=2$, $x=1/2$. Let
\begin{equation}
F_4(x,q)=\sum_{n=0}^{\infty} \frac{\left( \frac12+x \right)_n^3}
{(1+x)_n^3} z_4^{n+x}, \quad G_4(x,q)=\sum_{n=0}^{\infty} \frac{\left( \frac12+x \right)_n^3}{(1+x)_n^3} \left[ a_4+b_4(n+x) \right] z_4^{n+x}.
\end{equation}
The following identities hold:
\begin{equation}\label{formu-z-xmitad}
\phi(q)= 8 \sqrt{q} \sum_{n=0}^{\infty} \sigma_3(2n+1) \frac{(-q)^n}{(2n+1)^3}, \quad 
F_4 \left(\frac12,q \right)=F_4(0,q) \phi(q),
\end{equation}
and
\begin{equation}\label{formu-zab-xmitad}
G_4\left(\frac12,q\right)=\frac{8\sqrt{q}}{\pi} \left( \sum_{n=0}^{\infty} \sigma_3(2n+1)\frac{(-q)^n}{(2n+1)^3} -  \frac{\ln|q|}{2} \sum_{n=0}^{\infty} \sigma_3(2n+1)\frac{(-q)^n}{(2n+1)^2} \right),
\end{equation}
where $q=\pm e^{- \pi \sqrt{r}}$, and $\sigma_3(n)$ is the sum of the cubes of the divisors of $n$.
\end{theorem}

\begin{proof}
Applying Theorem \ref{teor-yang}, we obtain 
\begin{equation}\label{dif-eq-q-xmitad}
\left( q \frac{d}{dq} \right)^3 \phi(q) = \frac18 F_4^2(0,q) \; \sqrt{1-z_4} \; \sqrt{z_4}, \qquad \phi(q) = \frac{F_4(1/2,q)}{F_4(0,q)}.
\end{equation}
But for $s=2$ we know that
\[
F_4(0,q)=\theta_3^4(q), \quad z_4(q)=4 \lambda(q)(1-\lambda(q)), \quad \lambda(q)=\frac{\theta_2^4(q)}{\theta_3^4(q)}. 
\]
Using the identity 
$\theta_3^4(q)=\theta_2^4(q)+\theta_4^4(q)$, we get
\[
\left( q \frac{d}{dq} \right)^3 \phi(q) =  \frac14 \theta_2^2(q)\theta_4^2(q) \left[ \theta_4^4(q)-\theta_2^4(q) \right] = \sqrt{q} f(q).
\]
Then, with OEIS (on line encyclopedia of integer sequences), we could identify the coefficient of $(-q)^n$ in the expansion of $f(q)$ as $\sigma_3(2n+1)$. Hence
\[
\left( q \frac{d}{dq} \right)^3 \phi(q) =\sqrt{q} \sum_{n=0}^{\infty} \sigma_3(2n+1)(-q)^n,
\]
which proves (\ref{formu-z-xmitad}). Finally, we only have to apply Theorem \ref{teor-y-rama} to arrive at (\ref{formu-zab-xmitad}).
\end{proof}
The following identity is known:
\[
\sqrt{q} \sum_{n=0}^{\infty} \sigma_3(2n+1)(-q)^n = \frac{i}{240} \left[ E_4(\sqrt{-q}) - 9 E_4(-q) + 8 E_4(q^2) \right],
\]
where $E_4(q)$ is the Eisenstein series
\[
E_4(q)= \frac{45}{\pi^4} \sum_{(n,m) \neq (0,0)} \frac{1}{(n+\tau m)^4}, \qquad q=e^{2\pi i \tau}.
\]
Hence
\[
\left( q \frac{d}{dq} \right)^3 \phi(q) = \frac{i}{240} \left[ E_4(\sqrt{-q}) - 9 E_4(-q) + 8 E_4(q^2) \right].
\]
We use it to prove the following theorem:

\begin{theorem}
If $q=-e^{-\pi \sqrt{r}}$ where $r \in \mathbb{Q^{+}}$ (case of alternating series), we have
\begin{equation}\label{form-alter-level-4}
G_4\left(\frac12, z\right) = i \frac{r^{3/2}}{\pi^2} \left[ \frac{1}{16} S\left(1,0,\frac{r}{16}; 2\right) - \frac{9}{16} S\left(1,0,\frac{r}{4}; 2\right) + \frac12 S\left(1,0,r; 2\right) \right],
\end{equation}
and if $q=e^{-\pi \sqrt{r}}$ (case of series of positive terms), we have
\begin{equation}\label{form-plus-level-4}
G_4\left(\frac12, z\right) \! = \! \frac{r^{3/2}}{\pi^2} \left[ \frac{1}{16} S\left(1,1,\frac{r}{16}\!+\!\frac14; 2\right) - \frac{9}{16} S\left(1,1,\frac{r}{4} \!+\! \frac14; 2\right) - \frac12 S\left(1,1,r\!+\!\frac14; 2 \right) \right],
\end{equation}
where 
\[
S(A, B, C; t) = \sum_{(n,m) \neq (0,0)} \frac{1}{(An^2+Bnm+Cm^2)^t},
\]
is the Epstein zeta function \cite{glasser-zucker}.
\end{theorem}

\begin{proof}
If $q=-e^{-\pi\sqrt{r}}$ then  $-q=e^{-\pi \sqrt{r}}$, and the value of $\tau$ corresponding to $-q$ is $\tau=i \sqrt{r}/2$. If we define
\[
U_{n,m}(r) = \frac{1}{(n+m \frac{i \sqrt{r}}{4})^4} - \frac{9}{(n+m \frac{i \sqrt{r}}{2})^4} + \frac{8}{(n+m i \sqrt{r})^4},
\]
then 
\[
E_4(\sqrt{-q}) - 9 E_4(-q) + 8 E_4(q^2) = \sum_{n,n\neq(0,0)} U_{n,m}(r),
\]
and taking into account that $dq/q=\pi dr / (2\sqrt{r})$, we have
\[
\phi(r) \! = \! \frac{3i}{16\pi^5} \sum_{(n,m) \neq (0,0)} {\rm Re} \left[ \int \frac{\pi dr}{2\sqrt{r}} \int \frac{\pi dr}{2\sqrt{r}} \int \frac{\pi dr}{2\sqrt{r}} \, U_{n,m}(r) + \pi \sqrt{r} \int \frac{\pi dr}{2\sqrt{r}} \int \frac{\pi dr}{2\sqrt{r}} \, U_{n,m}(r) \right],
\]
where we have taken the real part inside the summation because for alternating series $\phi(r)$ has to be a purely imaginary number. Integrating and simplifying, we obtain (\ref{form-alter-level-4}). The proof of (\ref{form-plus-level-4}) is completely similar.
\end{proof}

\subsection{Examples of Ramanujan-type series shifted $1/2$ (level $\ell=4$)}
For $r=4$, using the known values:
\[
S(1,0,1; 2) = \frac{2\pi^2}{3} L_{-4}(2), \quad S(1,0,4; 2) = \frac{7\pi^2}{24}L_{-4}(2),
\]
see the method and the tables of \cite{glasser-zucker} or \cite{Bo-lattice}, and the obvious relation $S(1,0,\frac14; 2) = 16 S(1,0,4; 2)$, we get from (\ref{form-alter-level-4}):
\[
\sum_{n=0}^{\infty} \frac{(1)_n^3}{(\frac32)_n^3} \left( - \frac18 \right)^{n+\frac12} \left( \frac{4}{2\sqrt 2} + \frac{6}{2 \sqrt 2} n \right) = \frac{8 i}{\pi^2} \left( \frac32 S(1,0,4; 2) - \frac{9}{16} S(1,0,1; 2) \right)= \frac{i}{2} L_{-4}(2),
\]
where $L_{-4}(2)=G$ (the Catalan constant). Hence
\[
\sum_{n=0}^{\infty} (-1)^n \frac{(1)_n^3}{(\frac32)_n^3} \frac{2 + 3 n}{8^n}  = 2 G.
\]
Below, we show two more examples
\[
\sum_{n=0}^{\infty} \frac{(1)_n^3}{(\frac32)_n^3} \left( \frac{42\sqrt 5 +30}{32}n +\frac{26\sqrt 5 + 14}{32} \right) \frac{1}{2^{6n+3}} \left( \frac{\sqrt 5 -1}{2} \right)^{8n+4}=\frac{\pi^2}{240},
\]
which corresponds to the value $q=e^{-\pi \sqrt{15}}$, and
\[
\sum_{n=0}^{\infty} (-1)^n \frac{(1)_n^3}{(\frac32)_n^3} \left( \frac{5\sqrt 2 -6}{4}n +\frac{4\sqrt 2 -5}{4} \right) \left( \frac{\sqrt 2 -1}{2} \right)^{3n} = 2L_{-4}(2)-\frac{\sqrt{2}}{2} L_{-8}(2),
\]
which corresponds to the value $q=-e^{-\pi \sqrt{8}}$. 
Observe that in \cite{Gui-Rog} we arrive at the results by relating ``divergent" series to convergent ones by means of the ``upside-down" transformation. In addition, observe that for the levels $\ell=1,2,3$ the two transformations (shift and ``upside-down") lead to completely different series.

\subsection{Some $q$-series corresponding to $s=2$ ($\ell=4$) with other shifts}
We have proved the following identity:
\begin{equation}
\left( q \frac{d}{dq} \right)^3 \frac{F_4(x,q)}{F_4(0,q)} = x^3 F_4^2(0,q) \; \sqrt{1-z_4} \; z_4^x, \quad \phi(q)=\frac{F_4(x,q)}{F_4(0,q)}.
\end{equation}
Hence, if we define
\[
f(x,q)=F_4^2(0,q) \; \sqrt{1-z_4} \; z_4^x = \theta_3^8(q) (1-2\lambda(q)) \big[4\lambda(q)(1-\lambda(q)) \big]^x,
\]
we have
\[
\phi(q) = x^3 \int_0^{q} \int_0^{q} \int_0^{q} f(x,q) \; \frac{dq}{q} \; \frac{dq}{q} \, \frac{dq}{q}.
\]
Finally, we obtain
\[
\sum_{n=0}^{\infty} z^{n+x}  \frac{ \left( \frac12+x \right)_n^3}{ (1+x)_n^3} (a+b(n+x))=\frac{1}{\pi} \left( \phi(q) - \ln|q| \, q \frac{d \phi(q)}{dq} \right),
\]
For $m=2,3,4,6,8,12,24$, the function 
\[
h_m(q) = 64^{-1/m} \, f(1/m, q^m) = 16^{-1/m} \, \theta_3^8(q^m) (1-2\lambda(q^m)) \big[ \lambda(q^m) (1-\lambda(q^m)) \big]^{\frac{1}{m}}
\]
has integer coefficients. Below we write the cases $m=2, 3, 4, 6, 8$:
\begin{align*}
h_{2}(q) & = q-28q^3+126q^5-344q^7+757q^9-1332q^{11}+2198q^{13}-3528 q^{15} + 4914 q^{17} \\ 
& \quad - 6860 q^{19} + 9632q^{21}- 12168 q^{23} + 15751 q^{25} - 20440 q^{27} + 24390 q^{29} - \cdots, \\
h_{3}(q) & = q-24q^4+20q^7+0 q^{10}-70q^{13}+192q^{16}+56q^{19}+0q^{22} - 125q^{25}-480q^{28} \\ 
& \quad - 308 q^{31}+ 0 q^{34}+110 q^{37} + 0 q^{40} - 520 q^{43} + 0 q^{46} + 57q^{49} + 1680 q^{52}  + \cdots, \\
h_{4}(q) & = q+22q^5-27q^9-18q^{13}-94q^{17}+0q^{21}+359 q^{25} - 130 q^{29}+0q^{33} +214 q^{37} \\ 
& \quad  -230 q^{41}-594 q^{45}-343 q^{49}+518 q^{53} + 0 q^{57} +830 q^{61}-396 q^{65} + \cdots, \\
h_{6}(q) & = q-20q^7-70q^{13}-56q^{19}-125q^{25}-308q^{31}+110q^{37}-520q^{43} + 57q^{49} \\ 
& \quad + 0 q^{55} + 182 q^{61}-880 q^{67} +1190 q^{73}-884q^{79} + 0 q^{85} - 1400 q^{91} + \cdots, \\
h_{8}(q) &=q-19q^{9}-90 q^{17}-125 q^{25}-200q^{33}-522q^{41}-343q^{49}+360q^{57}+0 q^{65} \\ 
& \quad -430q^{73} + 145q^{81} +1026q^{89}+1910q^{97}-270q^{113}+3669q^{121}+1368q^{129} \\
& \quad -2250 q^{137} +0q^{145} + 1710q^{153} + 0 q^{161}-2197 q^{169}+920 q^{177} +\cdots.
\end{align*} 
The cases $m=2,3,4,6$ are in OEIS \cite{oeis}, while the cases $8,12,24$ are not yet in it. We observe that the coefficient of $q^k$ multiplied by the coefficient of $q^j$, where $k-1$ and $j-1$ are multiples of $m$, equals the coefficient of $q^{kj}$ when $k$ and $j$ are coprime.

\section{The $q$-series for Ramanujan series shifted $1/2$. Cases $s=4, 6, 3$}

In \cite{Gui-rama-closely} I conjecture the value of (\ref{rama-G-x}) in cases when $z$, $b$, $a$ are algebraic numbers, and $x=1/2$. The observed results corresponding to $s=4$, $s=6$ and $s=3$ involve neperian logarithms in case of alternating series and arc tangent values in case of series of positive terms. We rewrite those conjectures, together with all the other cases corresponding to rational values of $z$, in the tables of this paper. Some few but representative examples are in my thesis \cite[pp. 44--46]{Gui-thesis}. Notice that in \cite{Gui-rama-closely} and in \cite[pp. 44--46]{Gui-thesis} there are also examples of shifted series corresponding to Ramanujan-like series for $1/\pi^2$ and $1/\pi^3$. However we do not know how to get $q$-series for those shifted series.

\subsection{The $q$-series for Ramanujan series with $s=4$ ($\ell=2$) and the shift $1/2$}

\begin{theorem}
Case $s=4$, $x=1/2$. Let
\begin{multline}\label{desarrollo-2}
f(q)=\frac{1}{2\sqrt{q}} \int \eta^4(q) \eta^4(q^2) \left( 1-\frac{128}{64+\eta^{24}(q) \eta^{-24}(q^2)} \right) \frac{dq}{q} \\ = 1-44q+1126q^2-27096q^3+640909q^4-15036548q^5+351245038 q^6 - \cdots,
\end{multline}
and
\begin{align}
F_2(x,q) &=\sum_{n=0}^{\infty} \frac{\left( \frac12+x \right)_n\left( \frac14+x \right)_n\left( \frac34+x \right)_n}
{(1+x)_n^3} z_2^{n+x}, \\
G_2(x,q) &=\sum_{n=0}^{\infty} \frac{\left( \frac12+x \right)_n\left( \frac14+x \right)_n\left( \frac34+x \right)_n}{(1+x)_n^3} \left[a_2+b_2(n+x) \right] z_2^{n+x}, 
\end{align}
where $z$, $a$, $b$ depend on $q$ and $G(0,q)=1/\pi$. Then, the following identities hold:
\begin{equation}
F_2\left(\frac12,q\right)= 16 F_2(0,q) \sqrt{q} \sum_{n=0}^{\infty} c_n \frac{q^n}{(2n+1)^2},
\end{equation}
and
\begin{equation}
G_2\left(\frac12,q\right) = \frac{16\sqrt{q}}{\pi} \left( \sum_{n=0}^{\infty} c_n \frac{q^n}{(2n+1)^2} - \frac{\ln|q|}{2} \sum_{n=0}^{\infty} c_n \frac{q^n}{2n+1} \right), \label{cuartos-G}
\end{equation}
where $c_n$ is the coefficient of $q^n$ in $f(q)$.
\end{theorem}

\begin{proof}
In this case we know that \cite[Table 1]{ChanChanLiu}
\[
z_2(q)=4x_2(q)(1-x_2(q)), \quad F_2(0,q)=8\frac{\eta^8(q^2)}{\eta^4(q)} \frac{1}{\sqrt{x_2(q)}}, \quad 
x_2(q)=\frac{64}{64+ \eta^{24}(q)\eta^{-24}(q^2)},
\]
where $\eta(q)$ is the Dedekind $\eta$ function:
\[
\eta(q)=q^{1/24} \prod_{n=1}^{\infty} (1-q^n).
\]
Hence, for
\[
\left( q \frac{d}{dq} \right)^3 \phi(q)= \frac18 F_2^2(0,q) \sqrt{z_2(q)} \sqrt{1-z_2(q)},
\]
we obtain
\[
\left( q \frac{d}{dq} \right)^3 \phi(q) = \eta^4(q) \eta^4(q^2) \left( 1-\frac{128}{64+\eta^{24}(q) \eta^{-24}(q^2)} \right) = \sqrt{q} \, g(q),
\]
where
\begin{equation}\label{forma-s-4}
g(q)=1-3\cdot 44 \, q+ 5\cdot 126 \, q^2- 7 \cdot 27096 \, q^3+ 9\cdot 640909 \, q^4-11\cdot 15036548 \, q^5+\cdots.
\end{equation}
Hence, the theorem holds. 
\end{proof}

\begin{conjecture}
The coefficient of $q^n$ in $g(q)$ is divisible by $2n+1$ which is equivalent to assume that all the coefficients $c_n$ of $f(q)$ are integer numbers. In addition, $c_n \equiv 1 \pmod{p^2}$ when $2n+1$ is a prime number $p$.
\end{conjecture}

\subsection{The $q$-series for Ramanujan series with $s=3$ ($\ell=3$) shifted $1/2$}
\begin{theorem}
Let
\begin{multline}
f(q)=\frac{1}{2\sqrt{q}} \int \eta^2(q^3) \eta^6(q) \left( 1+9\frac{\eta^3(q^9)}{\eta^3(q)}  \right) \left( 1-\frac{54}{27+\eta^{12}(q) \eta^{-12}(q^3)} \right) \frac{dq}{q} \\ =
1-17 \, q + 126 \, q^2-832 \, q^3 + 5329 \, q^4 -  33516 \, q^5 + 209054 q^6 - 1298142 q^7 + \cdots,
\end{multline}
and
\begin{align*}
F_3(x,q) &=\sum_{n=0}^{\infty} \frac{\left( \frac12+x \right)_n\left( \frac13+x \right)_n\left( \frac23+x \right)_n}{(1+x)_n^3} z_3^{n+x}, \\
G_3(x,q) &=\sum_{n=0}^{\infty} \frac{\left( \frac12+x \right)_n\left( \frac13+x \right)_n\left( \frac23+x \right)_n}{(1+x)_n^3} \left[a_3+b_3(n+x) \right] z_3^{n+x}, 
\end{align*}
The following identities hold:
\begin{equation}
F_3\left(\frac12,q\right)= 16 F_3(0,q) \sqrt{q} \sum_{n=0}^{\infty} c_n \frac{q^n}{(2n+1)^2},
\end{equation}
and
\begin{equation}
G_3\left(\frac12,q\right) = \frac{16\sqrt{q}}{\pi} \left( \sum_{n=0}^{\infty} c_n \frac{q^n}{(2n+1)^2} - \frac{\ln|q|}{2} \sum_{n=0}^{\infty} c_n \frac{q^n}{2n+1} \right), \label{tercios-G}
\end{equation}
where $c_n$ is the coefficient of $q^n$ in $f(q)$.
\end{theorem}

\begin{proof}
For this case we know that \cite[Table 1]{ChanChanLiu}
\[
z_3(q)=4x_3(q)(1-x_3(q)), \quad \text{where} \quad x_3(q)=\frac{27}{27 + \eta^{12}(q) \eta^{-12}(q^3)},
\]
and
\[
F_3^2(0, q)=27 \eta^8(q^3) \left( 1+9 \frac{\eta^3(q^9)}{\eta^3(q)} \right) \left(1+\frac{1}{27} \frac{\eta^{12}(q)}{\eta^{12}(q^3)} \right),
\]
Hence, for
\[
\left( q \frac{d}{dq} \right)^3 \phi(q)= \frac18 F_3^2(0,q) \sqrt{z_3(q)} \sqrt{1-z_3(q)},
\]
we obtain
\[
\left( q \frac{d}{dq} \right)^3 \phi(q) = \eta^2(q^3) \eta^6(q) \left( 1+9\frac{\eta^3(q^9)}{\eta^3(q)}  \right) \left( 1-\frac{54}{27+\eta^{12}(q) \eta^{-12}(q^3)} \right) = \sqrt{q} \, g(q).
\]
where
\begin{equation}\label{forma-s-3}
g(q)=1-3 \cdot 17 \, q + 5 \cdot 126 \, q^2-7 \cdot 832 \, q^3 + 9 \cdot 5329 \, q^4 - 11\cdot 33516 \, q^5 + \cdots.
\end{equation}
Hence the theorem holds.
\end{proof}

\begin{conjecture}
The coefficient of $q^n$ in $g(q)$ is divisible by $2n+1$ which is equivalent to assume that all the coefficients $c_n$ of $f(q)$ are integer numbers. In addition, $c_n \equiv 1 \pmod{p^2}$ when $2n+1$ is a prime number $p$.
\end{conjecture}

\subsection{The $q$-series for Ramanujan series with $s=6$ ($\ell=1$) shifted $1/2$}

In this case we know that \cite[Table 1]{ChanChanLiu}
\[
F_1^2(0,q)=E_4(q), \quad z_1(q)=1728 \, \frac{\eta^{24}(q)}{E_4^3(q)},
\]
where $E_4(q)$ is the Eisenstein series
\[
E_4(q) = 1+240 \sum_{n=0}^{\infty}\sigma_3(n) q^n = 1+240 \sum_{n=1}^{\infty} \frac{n^3 q^n}{1-q^n}.
\]
Proceeding in the same way as in the other cases, we obtain
\begin{align*}
\left( q \frac{d}{dq} \right)^3 \phi(q) &= 3 \sqrt 3 \frac{\eta^{12}(q)}{E_8(q)} \sqrt{E_4^3(q)-1728 \eta^{24}(q)} = 3 \sqrt 3 \, \frac{\eta^{12}(q)}{E_4^2(q)} E_6(q)  
= 3 \sqrt 3 \,\eta^{12}(q) \frac{E_6(q)}{E_8(q)} \\ &= 3 \sqrt 3 \sqrt{q} ( 1-3 \cdot 332 q + 5 \cdot 81126 q^2 
 -7 \cdot 19147288 q^3 + 9 \cdot 4472942221 q^4 \\ & \hskip 2cm - 11\cdot 1040187455460 q^5 + \cdots),    
\end{align*}
where $E_6(q)$ and $E_8(q)$ are the Eisenstein series
\[
E_6(q) = 1-504 \sum_{n=0}^{\infty}\sigma_5(n) q^n = 1-504 \sum_{n=1}^{\infty} \frac{n^5 q^n}{1-q^n},
\]
and
\[
E_8(q) = 1 + 480 \sum_{n=0}^{\infty}\sigma_7(n) q^n = 1 + 480 \sum_{n=1}^{\infty} \frac{n^7 q^n}{1-q^n}.
\]
Hence
\begin{equation}
G_1(q)=\frac{3\sqrt 3}{\pi} \left[ \int \frac{dq}{q} \int  \frac{dq}{q} \int \frac{dq}{q} \, \eta^{12}(q)\frac{E_6(q)}{E_8(q)} - \ln |q| \int \frac{dq}{q} \int \frac{dq}{q} \eta^{12}(q)\frac{E_6(q)}{E_8(q)} \right].    
\end{equation}
For the first single integral, we have
\begin{align*}
\int \eta^{12}(q) \frac{E_6(q)}{E_8(q)} \frac{dq}{q} = 2 \sqrt{q} f(q) &= 2 \sqrt{q} \left(1-332 q + 81126 q^2-19147288 q^3 + 4472942221 q^4 \right. \\ & \left.  \quad - 1040187455460 q^5 + \cdots \right),
\end{align*}
and again, we observe that the coefficients $c_n$ of $q^n$ in $f(q)$ are all integer numbers, and that $c_n \equiv 1 \pmod{p^2}$ when $2n+1$ is a prime number $p$.

\section{Examples of conjectured formulas, $\ell=1,2,3$}
In this section we show several examples of evaluation of some Ramanujan-type series with a shift $x_0=1/2$. More examples are in the tables. For discovering the conjectured results we have used techniques of Experimental Mathematics, e.g. the PSLQ algorithm and the function identify.
For level $\ell=2$ and $q=-e^{-\pi \sqrt{13}}$:
\[
\sum_{n=0}^{\infty} \frac{\left( 1 \right)_n\left( \frac34  \right)_n\left( \frac54 \right)_n}{(\frac32)_n^3} \left(\frac{153}{72}+\frac{260}{72}n \right) \frac{(-1)^n}{18^{2n+1}} \, {\overset ? =} \, 2 \ln 3 -3 \ln 2,
\]
For level $\ell=2$ and $q=e^{-\pi \sqrt{58}}$:
\[
\sum_{n=0}^{\infty} \frac{\left( 1 \right)_n\left( \frac34  \right)_n\left( \frac54 \right)_n}{(\frac32)_n^3} \left(\frac{4 \cdot 14298}{9801\sqrt 2}+\frac{4 \cdot 26390}{9801\sqrt 2}n \right) \frac{1}{99^{4n+2}} \, {\overset ? =} \, \frac{13}{2}{\pi}-16\arctan \frac{\sqrt{2}}{2}-24\arctan \frac{\sqrt 2}{3}.
\]
It is interesting to observe that the last result can also be written with logarithms as
\[
-13 i \ln \frac{1+i}{1-i} + 8 i \ln \frac{\sqrt 2 + i}{\sqrt 2 -i} + 12 i \ln \frac{3+\sqrt{2}i}{3-\sqrt{2}i},
\]
and observe in addition that $(\sqrt 2 +i)(\sqrt 2 - i)=3$ and $(3+\sqrt2 \, i)(3-\sqrt 2 \, i)=11$, which are divisors of $99$. 
For level $\ell=3$ and $q=-e^{-\pi \sqrt{25/3}}$:
\[
\sum_{n=0}^{\infty} \frac{\left( 1 \right)_n\left( \frac56  \right)_n\left( \frac76 \right)_n}{(\frac32)_n^3} \left(\frac{11}{24} + \frac34 n \right) \frac{(-1)^n}{80^n} \, {\overset ? =} \, 9 \ln 3 - 2\ln 2 -5 \ln 5,
\]
For level $\ell=1$ and $q=e^{-\pi \sqrt{8}}$:
\[
\sum_{n=0}^{\infty} \frac{\left( 1 \right)_n\left( \frac23  \right)_n\left( \frac43 \right)_n}{(\frac32)_n^3} \left( \frac{136}{125} + \frac{224}{125} n \right) \left( \frac35 \right)^{3n}\, {\overset ? =} \, \pi - 4 \arctan \frac12.
\]
In the tables we show all the examples corresponding to rational values of $z$. We finally give an example of an irrational series. For level $\ell=2$ and $q=-e^{-\pi \sqrt{21}}$:
\begin{multline}\nonumber
\sum_{n=0}^{\infty} 
\frac{\left( 1 \right)_n\left( \frac34  \right)_n\left( \frac54 \right)_n}{(\frac32)_n^3}  
[(756+448\sqrt{3})n +(429+256 \sqrt{3})] \frac{(-1)^n}{(42+24 \sqrt{3})^{2n}} \\
\, {\overset ? =} \, 2 \cdot (42+24 \sqrt{3})^2 \cdot \ln \left[ \frac{42+24 \sqrt{3}}{81} \right]^2.
\end{multline}
Of course our conjectured evaluations agree with the numerical approximations obtained from the corresponding $G_{\ell}(1/2,q)$.

In the table $1$ we show the Ramanujan-type series for $1/\pi$ with rational values of $z$ in the case $s=4$ (level $2$) , and in the table 2 we have written the corresponding conjectured values of $G_2(1/2,q)$. In the tables $3$ and $5$ we show the Ramanujan-type series for $1/\pi$ with rational values of $z$ in the cases $s=3$ (level $3$) and $s=6$ (level $1$) respectively, and in the tables 4 and 6 we have written the corresponding conjectured values of $G_3(1/2,q)$ and $G_1(1/2,q)$ (the tables are in the last pages of the paper after the references).

\subsection{Conclusion}
May be that discovering explicit formulas (as we have done in (\ref{formu-zab-xmitad}) for the case $s=2$ and $x=1/2$) for the coefficients $c_n$ could be a useful step towards proving the patterns observed by the author. The final step would be evaluating the $q$-series at $q=\pm \exp(-\pi \sqrt{r})$. The analog patterns observed for shifted Ramanujan-like series for $1/\pi^k$ with $k \geq 2$ see \cite{Gui-rama-closely} and  \cite[pp. 44--46]{Gui-thesis}  are further beyond the ideas of this paper.

\begin{table}[p]
    \begin{tabular}{|c c c c | c c c c|}
        \hline &&&&&&& \\
        $q$ & $a$ & $b$ & $z<0$ & $q$ & $a$ & $b$ & $z>0$ \\ &&&&&&& \\
        \hline \hline &&&&&&& \\
        $-e^{-\pi \sqrt{5}}$ & $\frac{3}{8}$ & $\frac{20}{8}$ & $-\frac{1}{4}$ &
        $e^{-\pi  \sqrt{4}}$ & $\frac{2}{9}$ & $\frac{14}{9}$ & $\frac{32}{81}$ \\ &&&&&&& \\
        $-e^{-\pi \sqrt{7}}$ & $\frac{8}{9\sqrt7}$ & $\frac{65}{9\sqrt7}$ & $-\frac{16^2}{63^2}$ &
        $e^{-\pi  \sqrt{6}}$ & $\frac{1}{2\sqrt3}$ & $\frac{8}{2\sqrt3}$ & $\frac{1}{9}$ \\ &&&&&&& \\
        $-e^{-\pi \sqrt9}$ & $\frac{3\sqrt3}{16}$ & $\frac{28\sqrt3}{16}$ & $-\frac{1}{48}$ &
        $e^{-\pi  \sqrt{10}}$ & $\frac{4}{9\sqrt2}$ & $\frac{40}{9\sqrt2}$ & $\frac{1}{81}$ \\ &&&&&&& \\
        $-e^{-\pi \sqrt{13}}$ & $\frac{23}{72}$ & $\frac{260}{72}$ & $-\frac{1}{18^2}$ &
        $e^{-\pi  \sqrt{18}}$ & $\frac{27}{49\sqrt3}$ & $\frac{360}{49\sqrt3}$ & $\frac{1}{7^4}$ \\ &&&&&&& \\
        $-e^{-\pi \sqrt{25}}$ & $\frac{41\sqrt5}{288}$ & $\frac{644\sqrt5}{288}$ & $-\frac{1}{5 \cdot 72^2}$ &
        $e^{-\pi  \sqrt{22}}$ & $\frac{19}{18\sqrt{11}}$ & $\frac{280}{18\sqrt{11}}$ & $\frac{1}{99^2}$ \\ &&&&&&& \\
        $-e^{-\pi \sqrt{37}}$ & $\frac{1123}{3528}$ & $\frac{21460}{3528}$ & $-\frac{1}{882^2}$ &
        $e^{-\pi  \sqrt{58}}$ & $\frac{4412}{9801\sqrt2}$ & $\frac{105560}{9801\sqrt2}$ & $\frac{1}{99^4}$ \\ &&&&&&& \\
        \hline
    \end{tabular}
    \vskip 0.5cm
    \caption{Ramanujan series with $s=4$ ($\ell=2$)}
\end{table}

\begin{table}[b]
    \begin{tabular}{|c c | c c |}
        \hline &&& \\
        $q$ & $-i \, G_2(\frac12,q)$ & 
        $q$ & $G_2(\frac12,q)$  \\ &&& \\
        \hline \hline &&& \\
        $-e^{-\pi \sqrt{5}}$ & $\ln 2$ & 
        $e^{-\pi  \sqrt{4}}$ & $\frac{\pi}{2}-2\arctan \frac{1}{2\sqrt 2}$ \\ 
        &&& \\ 
        $-e^{-\pi \sqrt{7}}$ & $\ln(88+13\sqrt 7)-4\ln 3$ 
        & $e^{-\pi  \sqrt{6}}$ & $\frac{\pi}{6}$ \\ 
        &&& \\
        $-e^{-\pi \sqrt{9}}$ & $\frac32 \ln 3 - 2 \ln 2$ & 
        $e^{-\pi  \sqrt{10}}$ & $\frac{\pi}{2}+4\arctan \frac{1}{2 \sqrt 2}$ \\ 
        &&& \\
        $-e^{-\pi \sqrt{13}}$ & $2\ln 3-3\ln 2$ & 
        $e^{-\pi  \sqrt{18}}$ & $-\frac{\pi}{6}+4\arctan \frac{1}{4\sqrt 3} $ \\ 
        &&& \\ 
        $-e^{-\pi \sqrt{25}}$ & $9\ln 2-2\ln 3-\frac52\ln 5$ & 
        $e^{-\pi  \sqrt{22}}$ & $-\frac{\pi}{2}+4\arctan \frac{7}{5\sqrt{11}} $ \\ 
        &&& \\
        $-e^{-\pi \sqrt{37}}$ & $\ln 2+10\ln 3-6 \ln 7$ & 
        $e^{-\pi  \sqrt{58}}$ & $\frac{13\pi}{2}-16\arctan \frac{1}{\sqrt 2}-24\arctan \frac{\sqrt 2}{3}$ \\ 
        &&& \\
        \hline
    \end{tabular}
    \vskip 0.5cm
    \caption{Some conjectured values of $G_2(1/2,q)$}
\end{table}

\begin{table}[ht]
    \begin{tabular}{|c c c c | c c c c|}
        \hline &&&&&&& \\
        $q$ & $a$ & $b$ & $z<0$ & $q$ & $a$ & $b$ & $z>0$ \\ &&&&&&& \\
        \hline \hline &&&&&&& \\
        $-e^{-\pi \sqrt{9/3}}$ & $\frac{\sqrt{3}}{4}$ & $\frac{5\sqrt{3}}{4}$ & $-\frac{9}{16}$ &
        $e^{-\pi  \sqrt{8/3}}$ & $\frac{1}{3\sqrt{3}}$ & $\frac{6}{3\sqrt{3}}$ & $\frac{1}{2}$ \\ &&&&&&& \\
        $-e^{-\pi \sqrt{17/3}}$ & $\frac{7}{12\sqrt3}$ & $\frac{51}{12\sqrt3}$ & $-\frac{1}{16}$ &
        $e^{-\pi  \sqrt{16/3}}$ & $\frac{8}{27}$ & $\frac{60}{27}$ & $\frac{2}{27}$ \\ &&&&&&& \\
        $-e^{-\pi \sqrt{25/3}}$ & $\frac{\sqrt{15}}{12}$ & $\frac{9\sqrt{15}}{12}$ & $-\frac{1}{80}$ &
        $e^{-\pi  \sqrt{20/3}}$ & $\frac{8}{15\sqrt3}$ & $\frac{66}{15\sqrt3}$ & $\frac{4}{125}$ \\ &&&&&&& \\
        $-e^{-\pi \sqrt{41/3}}$ & $\frac{106}{192\sqrt3}$ & $\frac{1230}{192\sqrt3}$ & $-\frac{1}{2^{10}}$
        &&&& \\ &&&&&&& \\
        $-e^{-\pi \sqrt{49/3}}$ & $\frac{26\sqrt7}{216}$ & $\frac{330\sqrt7}{216}$ & $-\frac{1}{3024}$
        &&&& \\ &&&&&&& \\
        $-e^{-\pi \sqrt{89/3}}$ & $\frac{827}{1500\sqrt3}$ & $\frac{14151}{1500\sqrt3}$ & $-\frac{1}{500^2}$
        &&&& \\ &&&&&&& \\
        \hline
    \end{tabular}
    \vskip 0.5cm
    \caption{Ramanujan series for $s=3$ ($\ell=3$)}
\end{table}

\begin{table}[ht]
    \begin{tabular}{|c c | c c |}
        \hline &&& \\
        $q$ & $-i \, G_3(\frac12,q)$ & 
        $q$ & $G_3(\frac12,q)$  \\ &&& \\
        \hline \hline &&& \\
        $-e^{-\pi \sqrt{9/3}}$ & $\frac{\sqrt 3}{4} \left( 3\ln 3 - 2\ln 2 \right) $ & 
        $e^{-\pi  \sqrt{8/3}}$ & $\frac{\sqrt 3}{4} \left( 3\pi-12\arctan \frac{\sqrt 2}{2} \right)$ \\ 
        &&& \\ 
        $-e^{-\pi \sqrt{17/3}}$ & $\frac{3\sqrt 3}{4} \left( 2\ln 2 - \ln 3 \right)$ 
        & $e^{-\pi  \sqrt{16/3}}$ & $\frac{\sqrt 3}{4} \left( 5\pi-24 \arctan \frac{\sqrt 2}{2} \right)$ \\ 
        &&& \\
        $-e^{-\pi \sqrt{25/3}}$ & $\frac{\sqrt{3}}{4} \left( 9\ln 3-2\ln 2-5\ln 5 \right) $ & 
        $e^{-\pi  \sqrt{20/3}}$ & $\frac{\sqrt 3}{4} \left( -3\pi + 12 \arctan \frac{\sqrt 5}{2} \right)$ \\ 
        &&& \\
        $-e^{-\pi \sqrt{41/3}}$ & $\frac{3\sqrt 3}{4} \left( 8\ln 2 - 5\ln 3 \right)$  & 
        $$ & $$ \\ 
        &&& \\ 
        $-e^{-\pi \sqrt{49/3}}$ & $\frac{\sqrt 3}{4} \left( 7\ln 7-10\ln 2-6\ln 3\right)  $ & 
        $$ & $$ \\ 
        &&& \\
        $-e^{-\pi \sqrt{89/3}}$ & $\frac{3\sqrt{3}}{4}  \left( 6 \ln 5 -6\ln 2 -5\ln 3\right)
        $ & $$ & $$ \\ 
        &&& \\
        \hline
    \end{tabular}
    \vskip 0.5cm
    \caption{Some conjectured values of $G_3(1/2,q)$}
\end{table}

\begin{table}[ht]
    \begin{tabular}{|c c c c | c c c c|}
        \hline &&&&&&& \\
        $q$ & $a$ & $b$ & $z<0$ & $q$ & $a$ & $b$ & $z>0$ \\ &&&&&&& \\
        \hline \hline &&&&&&& \\
        $-e^{-\pi\sqrt{7}}$ & $\frac{8}{5\sqrt{15}}$ & $\frac{63}{5\sqrt{15}}$ & $-\frac{4^3}{5^3}$ &
        $e^{-\pi \sqrt{8}}$ & $\frac{3}{5\sqrt5}$ & $\frac{28}{5\sqrt5}$ & $\frac{3^3}{5^3}$ \\ &&&&&&& \\
        $-e^{-\pi\sqrt{11}}$ & $\frac{15}{32\sqrt2}$ & $\frac{154}{32\sqrt2}$ & $-\frac{3^3}{8^3}$ &
        $e^{-\pi \sqrt{12}}$ & $\frac{6}{5\sqrt{15}}$ & $\frac{66}{5\sqrt{15}}$ & $\frac{4}{5^3}$ \\ &&&&&&& \\
        $-e^{-\pi\sqrt{19}}$ & $\frac{25}{32\sqrt6}$ & $\frac{342}{32\sqrt6}$ & $-\frac{1}{8^3}$ &
        $e^{-\pi \sqrt{16}}$ & $\frac{20}{11\sqrt{33}}$ & $\frac{252}{11\sqrt{33}}$ & $\frac{2^3}{11^3}$ \\ &&&&&&& \\
        $-e^{-\pi\sqrt{27}}$ & $\frac{279}{160\sqrt{30}}$ & $\frac{4554}{160\sqrt{30}}$ & $-\frac{9}{40^3}$ &
        $e^{-\pi \sqrt{28}}$ & $\frac{144\sqrt3}{85\sqrt{85}}$ & $\frac{2394\sqrt{3}}{85\sqrt{85}}$ & $\frac{4^3}{85^3}$ \\ &&&&&&& \\
        $-e^{-\pi \sqrt{43}}$ & $\frac{526\sqrt{15}}{80^2}$ & $\frac{10836 \sqrt{15}}{80^2}$ & $-\frac{1}{80^3}$ &&&& \\ &&&&&&& \\
        $-e^{-\pi \sqrt{67}}$ & $\frac{10177\sqrt{330}}{3 \cdot 440^2}$ & $\frac{261702\sqrt{330}}{3\cdot 440^2}$ & $-\frac{1}{440^3}$ &&&& \\ &&&&&&& \\
        $-e^{-\pi \sqrt{163}}$ & $\frac{27182818\sqrt{10005}}{3 \cdot 53360^2}$ & $\frac{1090280268\sqrt{10005}}{3\cdot 53360^2}$ & $-\frac{1}{53360^3}$ &&&& \\ &&&&&&& \\
        \hline
    \end{tabular}
    \vskip 0.5cm
    \caption{Ramanujan series for $s=6$ ($\ell=1$)}
\end{table}

\begin{table}[ht]
    \begin{tabular}{|c c | c c |}
        \hline &&& \\
        $q$ & $-i \, G_1(\frac12,q)$ & 
        $q$ & $G_1(\frac12,q)$  \\ &&& \\
        \hline \hline &&& \\
        $-e^{-\pi \sqrt{7}} $ & $\frac{3\sqrt 3}{8} \ln \frac{3^3}{5}$ & 
        $e^{-\pi  \sqrt{8}}$ & $\frac{3\sqrt 3}{8} \left(\pi - 4 \arctan \frac12 \right)$ \\ 
        &&& \\ 
        $-e^{-\pi \sqrt{11}}$ & $\frac{3\sqrt 3}{8} \ln 2$ 
        & $e^{-\pi  \sqrt{12}}$ & $\frac{3\sqrt 3}{8} \left(-\pi + 8\arctan \frac12 \right)$ \\ 
        &&& \\
        $-e^{-\pi \sqrt{19}}$ & $\frac{3\sqrt{3}}{8} \ln \frac{2^5}{3^3}$ & 
        $e^{-\pi  \sqrt{16}}$ & $\frac{3\sqrt 3}{8} \left(3\pi - 4\arctan \frac{\sqrt 2}{3}-12\arctan \frac{\sqrt 2}{2} \right)$ \\ 
        &&& \\
        $-e^{-\pi \sqrt{27}}$ & $\frac{3\sqrt 3}{8} \ln \frac{3^3 \cdot 5}{2^7}$ & 
        $e^{-\pi \sqrt{28}}$ & $\frac{3\sqrt 3}{8} \left(3\pi - 16\arctan \frac12 -8\arctan \frac14 \right)$ \\ 
        &&& \\ 
        $-e^{-\pi \sqrt{43}}$ & $\frac{3\sqrt 3}{8} \ln \frac{2^2 \cdot 3^9}{5^7}$ & 
        $$ & $$ \\ 
        &&& \\
        $-e^{-\pi \sqrt{67}}$ & $\frac{3\sqrt 3}{8} \ln \frac{2^{13} \cdot 11^5}{3^3\cdot 5^{11}}$ & 
        $$ & $$ \\ 
        &&& \\
        $-e^{-\pi \sqrt{163}}$ & $\frac{3\sqrt 3}{8} \ln \frac{3^{21}\cdot 5^{13}\cdot 29^5}{2^{38}\cdot 23^{11}}$ & 
        $$ & $$ \\ 
        &&& \\
        \hline
    \end{tabular}
    \vskip 0.5cm
    \caption{Some conjectured values of $G_1(1/2,q)$}
\end{table}

\end{document}